\documentclass[12pt,twoside,reqno,psamsfonts]{amsart}
\usepackage{amssymb}
\makeindex
\def\R{{\mathbb R}}

\def\Z{{\mathbb Z}}
\def\N{{\mathbb N}}

\def\squareforqed{\hbox{\rlap{$\sqcap$}$\sqcup$}}
\def\qed{\ifmmode\squareforqed\else{\unskip\nobreak\hfil
\penalty50\hskip1em\null\nobreak\hfil\squareforqed
\parfillskip=0pt\finalhyphendemerits=0\endgraf}\fi}
\newtheorem{thm}{Theorem}[section]
\newtheorem{cor}[thm]{Corollary}
\newtheorem{lem}[thm]{Lemma}

\begin{document}
\title{Two-dimensional Shannon type expansions via one-dimensional affine and wavelet lattice actions}
\author{K. Nowak}
\address{\textnormal{Department of Computer Science,
Drexel University, 3141 Chestnut Street, Philadelphia, PA 19104, USA} \newline
\textnormal{e-mail: kn33@drexel.edu}}
\author{M. Pap}
\address{\textnormal{Faculty of Sciences, University of P\'ecs, 7634 P\'ecs, Ifj\'us\'ag \'ut 6, HUNGARY} \newline
\textnormal{e-mail: papm@gamma.ttk.pte.hu}}

\begin{abstract}
It is rather unexpected, but true, that it is possible to construct reproducing formulae and orthonormal bases of $L^2 (\R^2)$ just by applying the standard one dimensional wavelet action of translations and dilations to the first variable $x_1$ of the generating function $\psi(x_1,x_2)$, $\psi \in L^2 (\R^2)$, i.e., by making use of building blocks
$$\psi_{u,s}(x_1,x_2)=s^{-1/2}\psi\left(\frac{x_1-u}{s},x_2\right), \text{where } u\in \R, s>0,$$ 
in the case of reproducing formulae, and 
$$\psi_{k,m}(x_1,x_2)=2^{-k/2} \psi\left(\frac{x_1-2^k m}{2^k},x_2 \right),  \text{where } k,m\in \Z,$$
in the case of orthonormal bases. It is possible to compensate the fact, that the second variable $x_2$ is not acted upon, by a careful selection of the generating function $\psi$. Shannon wavelet  tiling of the time-frequency plane $\R^2$, a standard illustration of orthogonality and completeness phenomena corresponding to the Shannon wavelet,
$$
\chi_{(2^km,2^k(m+1)]}(x)
\chi_{2^{-k}I}(\xi),\,  k,m\in \Z, \,I=- (1/2,1]\cup (1/2,1],
$$
with $x$ representing time and $\xi$ frequency,
is substituted by a phase space tiling of $\R^4$ with unbounded, hyperboloid type blocks of the form
$$
\chi_{(2^km,2^k(m+1)]}(x_1)\sum_{n,l}\chi_{2^{-k}I_{D(n,l)}}(\xi_1)
\chi_{(n,n+1]}(x_2)\chi_{(l,l+1]}(\xi_2),\, k,m\in \Z
$$
where $I_r=2^{-r}I$, $ r\ge 1$, and $D:\Z \times \Z \rightarrow \N$ is a bijection, an 
additional parameter of the generating function, needed for the lift from $L^2(\R)$ to $L^2(\R^2)$. Variables $x_1,x_2$ are coordinates of position and variables $\xi_1,\xi_2$ of momentum.
\end{abstract}

\maketitle
\markboth{Shannon type expansions}{K. Nowak, M. Pap}

\section{Main Results and their Context}
We begin by introducing the $L^2(\R)$ background results. The one-dimensional {\it Calder\'on reproducing formula} has the following form: for a function $\psi \in L^2(\R)$ satisfying the {\it admissibility condition}, i.e.,
\begin{equation}
\int_0^\infty |\hat{\psi}(s)|^2\frac{ds}{s}=
\int_0^\infty |\hat{\psi}(-s)|^2\frac{ds}{s}=1,
\label{calderon_admissibility_condition}
\end{equation} 
where $\hat{\psi}(\xi)=\int_\R \psi(x)e^{-2\pi ix\xi}\,dx $
is the Fourier transform of $\psi$, we have    
\begin{equation}
f=\int_{\R^2_+} \langle f,\psi_{u,s}\rangle_{L^2(\R)} \psi_{u,s}\,\frac{du\,ds}{s^2}
\label{calderon_reproducing_formula}\end{equation}
for all $f\in L^2(\R)$,
where $R^2_+=\R \times (0,\infty)$, $\psi_{u,s}(x)=s^{-\frac{1}{2}}\psi(\frac{x-u}{s})$, $u\in \R,\,s>0$, and the convergence of the integral in (\ref{calderon_reproducing_formula}) is understood in the weak sense. Function $\psi$ is called the {\it generating function} of the Calder\'on reproducing formula (\ref{calderon_reproducing_formula}). The admissibility condition
(\ref{calderon_admissibility_condition}) is necessary and sufficient for the formula (\ref{calderon_reproducing_formula}) to hold. If 
(\ref{calderon_reproducing_formula}) holds we call the system $\{\psi_{u,s}\}_{u\in \R,s>0}$, together with its parameter measure $\frac{du\,ds}{s^2}$, {\it reproducing} in $L^2(\R)$.

The one-dimensional {\it Shannon wavelet system} is defined as
\begin{equation}
\psi_{k,m}^S(x)=2^{-k/2}\psi^S  \left(\frac{x-2^km}{2^k}\right),
\label{shannon_wavelet}
\end{equation}
where $k,m\in \Z$, and the Fourier transform $\widehat{\psi^S}$ of $\psi^S$ has the form $\widehat{\psi^S}=\chi_{[-1,-1/2)\cup (1/2,1]}$. The system $\left\{ \psi_{k,m}^S\right\}_{k,m\in \Z}$ is an orthonormal basis of $L^2(\R)$, because on the Fourier transform side it represents a family of the standard trigonometric systems adapted to the dyadic partition of the real line $\R$, i.e.
\begin{equation}
\widehat{\psi_{k,m}^S}(\xi)=2^{k/2}\chi_{[-1,-1/2)\cup (1/2,1]}
(2^k \xi)e^{2\pi im2^k \xi}.
\label{shannon_wavelet_fourier_side} 
\end{equation}

We move now to the $L^2(\R^2)$ context. We introduce $e_{k,l} (y)=\chi_{(k,k+1]}(y)\,e^{2\pi ily}$, with $k,l\in \Z$, and $f_m (s)=\chi_{(2^{-m-1},2^{-m} ]} (|s|)$, with $m \ge 1, m\in \Z$. The system $\left\{e_{k,l} \right\}_{k,l\in \Z}$ is an orthonormal basis of $L^2 (\R)$, and $\left\{c_f^{-1} f_m \right\}_{m\ge 1}$, where $c_f=(2 \log 2)^{1/2}$, is an orthonormal system on 
$L^2 (\R,\frac{ds}{|s|})$. Let $D:\Z \times \Z \rightarrow \N$ be a bijection. Define the generating function $\psi^D \in L^2 (\R^2)$ by requesting that 
\begin{equation}
\psi^D(\hat{s},y)=\sum_{k,l\in \Z} f_{D(k,l)}(s)e_{k,l} (y),
\label{2d_generating_function}
\end{equation}
where the symbol $\hat{\text{ }}$ over an indicated variable means that the Fourier transform was applied to it. 
Function $\psi ^D$ defined in (\ref{2d_generating_function})
is our principal object of interest. It is the generating function of the reproducing system
\begin{equation}
d_f^{-1}\psi_{u,s}^D(x_1,x_2)=d_f^{-1}s^{-\frac{1}{2}}\psi^D\left(\frac{x_1-u}{s},x_2\right)\!, u\in \R,\,s>0, 
\label{2d_reproducing_system}
\end{equation}
where $d_f=(\log 2)^{1/2}$, and of the orthonormal basis
\begin{equation}
\psi_{k,m}^D(x_1,x_2)=2^{-k/2}\psi ^D\left(\frac{x_1-2^km}{2^k},x_2\right)\!, k,m\in \Z.
\label{2d_orthonormal_basis}
\end{equation}
Both systems (\ref{2d_reproducing_system}) and (\ref{2d_orthonormal_basis}) are obtained by an application of the standard one-dimensional wavelet action defined in (\ref{calderon_reproducing_formula}) and (\ref{shannon_wavelet}) to the first coordinate $x_1$.

We are ready to formulate our principal results.
\begin{thm}\label{Reproducing_Formula_iff_1d_Isometry}
Let us consider $\psi\in L^2(\R^2)$. The system
$$
\psi_{u,s}(x_1,x_2)=s^{-\frac{1}{2}}\psi\left(\frac{x_1-u}{s},x_2\right)\!, u\in \R,\,s>0, 
$$
with the parameter measure $\frac{du\,ds}{s^2}$, is reproducing in $L^2(\R^2)$ if and only if the maps
$$
f\mapsto \int_\R \overline{\psi\left(\widehat{\pm s},y\right)} f(y)\,dy,
$$
from $L^2(\R)$ into $L^2(\R_+,\frac{ds}{s})$, preserve inner products.
\end{thm}

\begin{cor}\label{2d_Reproducing_System}
The system $\left\{d_f^{-1}\psi_{u,s}^D\right\}_{u\in \R,s>0}$, defined in (\ref{2d_reproducing_system}), with the parameter measure $\frac{du\,ds}{s^2}$, is reproducing in $L^2(\R^2)$, i.e. for all 
$f\in L^2(\R^2)$
$$
f=d_f^{-2}\int_{\R^2_+} \langle f,\psi_{u,s}^D\rangle_{L^2(\R^2)} \psi_{u,s}^D\,\frac{du\,ds}{s^2},
$$
with the convergence of the integral understood in the week sense.
\end{cor}

For a measurable function $f$, defined on a topological space $X$, equipped with a Borel measure $\mu$, we define its essential support $\text{ess-supp}\,f$ as the intersection of all closed sets $F$, satisfying 
$f(x)=0$ for $\mu$-almost every $x$ in the complement of $F$.  

\begin{thm}\label{Reproducing_Discrete_System_iff_1d_Isometry}
Let us consider $\psi\in L^2(\R^2)$. Suppose that for almost every $y \in \R$ $\text{ess-supp}\, \psi(\hat{\cdot},y)\subset [-1/2,1/2]$.
The system
$$
\psi_{k,m}(x_1,x_2)=2^{-k/2}\psi \left(\frac{x_1-2^km}{2^k},x_2\right)\!, k,m\in \Z ,
$$
with the parameter measure being the counting measure on $\Z\times \Z$,
is reproducing in $L^2(\R^2)$ if and only if for 
every pair $f,g\in L^2(\R)$ the equality
\begin{equation}
\langle f,g \rangle = \sum_k \int_\R 
\overline{\psi\left(\widehat{2^k \xi},y\right)} f(y)\,dy
\int_\R \psi\left(\widehat{2^k \xi},y\right) \overline{g(y)}\,dy
\label{discrete_isometry}
\end{equation}
holds for almost every $\xi\in \R$. 
\end{thm}

\begin{cor}\label{2d_Orthonormal_Basis}
The system $\left\{\psi_{k,m}^D\right\}_{k,m\in \Z}$, defined in 
(\ref{2d_orthonormal_basis}), is an orthonormal basis of $L^2(\R^2)$.
\end{cor}

Time-frequency representations exhibit the essence of wavelet orthonormal systems. The basic time-frequency characteristics of wavelets are frequently illustrated with tilings of the plane $\{(x,\xi)\,|\,x,\xi \in \R\}$, where coordinate  $x$ represents time and coordinate $\xi$ frequency. In the case of our current interest, Shannon wavelet  tiling has the form
\begin{equation}
\left\{ \chi_{(2^km,2^k(m+1)]}(x) \chi_{2^{-k}I}(\xi)\right\}_{ k,m\in \Z},\, I=-(1/2,1]\cup (1/2,1].
\label{shannon_wavelet_tiling}
\end{equation}
In terms of set representations, this is just the basic building block $(0,1]\times I$ transformed via affine actions of the lattice points $\left\{(2^km,2^{-k})\right\}_{k,m \in \Z}$, resulting in a partitioning of the time-frequency plane into pairs of rectangles of combined area $1$, pairs corresponding to a separate treatment of positive and negative frequencies. In terms of the interpretation, covering property corresponds to completeness, and disjointness of blocks of the partitioning to orthogonality relations.
How about the lift to $L^2(\R^2)$ presented in Corollary \ref{2d_Orthonormal_Basis}? Can one get an equally clear explanation of completeness and orthogonality phenomena for $\left\{\psi_{k,m}^D\right\}_{k,m\in \Z}$  via affine actions of the lattice points $\left\{(2^km,2^{-k})\right\}_{k,m \in \Z}$? Yes, there is a very similar explanation, but now the building blocks of the partitioning are unbounded, of a  hyperboloid type, with their exact shape depending on the choice of the bijection $D$. The exact form of the  tiling of the phase space $\R^4=\left\{ (x_1,x_2,\xi_1,\xi_2)\,|\, x_1,x_2, \xi_1,\xi_2 \in \R \right\}$, with $x_1,x_2$ position and $\xi_1,\xi_2$ momentum coordinates, is now
\begin{equation}
\chi_{(2^km,2^k(m+1)]}(x_1)\sum_{n,l}
\chi_{(n,n+1]}(x_2)\chi_{(l,l+1]}(\xi_2) \chi_{2^{-k}I_{D(n,l)}}(\xi_1),\, k,m\in \Z,
\label{phase_space_tiling_interpretation}
\end{equation}
with $I_r=2^{-r}I$, $ r\ge 1$, and $D:\Z \times \Z \rightarrow \N$ a bijection. The expression (\ref{phase_space_tiling_interpretation}) is a product of two functions, the first one depending on $x_1$ and the second one on $x_2,\xi_2,\xi_1$. The first factor is one dimensional and it is the same as in (\ref{shannon_wavelet_tiling}).  The second factor is three dimensional.  It is convenient to visualize it  as a graph over the horizontal plane $\{(x_2,\xi_2)\,|\, x_2,\xi_2 \in \R\}$, partitioned into squares $(n,n+1]\times (l,l+1]$, $n,l \in \Z$, with two boxes $(n,n+1]\times (l,l+1]\times 2^{-k}I_{D(n,l)}$, one
placed above and one placed below each square. Variable $\xi_1$ is the vertical coordinate.

The problem of constructing reproducing formulae out of group representations attracted substantial attention in the field of wavelets  since early 90's. The origin of it goes back to the theory of coherent states of mathematical physics. The book by Ali-Antoine-Gazeau \cite{AAG} presents both the current stage of development, as well as the background results. In the field of wavelets, the analytic techniques needed in the study, come from representation theory and analysis in phase space. A large variety of applications provides a strong motivation for the extensive study of the topic. Despite many efforts, no comprehensive understanding of constructions and classifications of reproducing formulae has been achieved so far. New constructions emerge often. In chronological order, substantial contributions came from Torr\'esani \cite{Tor1}, \cite{Tor2}, Kalisa-Torr\'esani \cite{KaTo},  Hogan-Lakey \cite{HoLa}, Bernier-Taylor \cite{BeTa}, Laugesen-Weaver-Weiss-Wilson \cite{LWWW}, De Mari-Nowak \cite{DMNo1}, \cite{DMNo2}, F\"uhr \cite{Fuh}, Cordero-De Mari-Nowak-Tabacco \cite{CDMNT1}, \cite{CDMNT2}, \cite{CDMNT3} , Dahlke-Steidl-Teschke \cite{DST}, De Mari-De Vito and collaborators \cite{DM&Co1}, \cite{DM&Co2}, \cite{DM&Co3}, \cite{DM&Co4}, Cordero-Tabacco \cite{CoTa}, Czaja-King \cite{CzKi1}, \cite{CzKi2}, Namngam-Schulz \cite{NaSc1}, \cite{NaSc2}. One of the reproducing formulae listed in \cite{DM&Co2}, \cite{DM&Co3} gave the inspiration for our current study. The classical paper by Fefferman \cite{Fef} beatifully presents various aspects of the usage of phase space tilings with Heisenberg boxes. The range of applicability, from the point of view of analysis going beyond the context of spaces of square integrable functions defined on Euclidean spaces, of our phase space partitioning into hyperboloid type blocks, still has to be identified.

Books by Daubechies \cite{Dau}, Gr\"ochenig \cite{Gro}, Folland \cite{Fol}, Wojtaszczyk \cite{Woj} are comprehensive references on phase-space analysis and wavelets. We refer the reader to books by {\L}ojasiewicz \cite{Loj} and Rudin \cite{Rud} for the background results we use in our proofs.

\section{Proofs of the Main Results and Auxiliary Facts}

We use the Fourier transform in the form 
$$
\mathcal{F}f(\xi) =\hat{f}(\xi)=\int_{\R^d}f(x)\,e^{-2\pi i \langle x,\xi\rangle}\,dx.
$$
Fourier transform $\mathcal{F}$ is a bijective map on the Schwartz class $\mathcal{S}(\R^d)$, and its inverse $\mathcal{F}^{-1}$ is represented by the integral
$$
\mathcal{F}^{-1}f(\xi)=\check{f}(\xi)=\int_{\R^d}f(x)\,e^{2\pi i \langle x,\xi\rangle}\,dx.
$$
It extends, by density, from the Schwartz class $\mathcal{S}(\R^d)$ to a unitary operator defined on the Hilbert space $L^2(\R^d)$, and, by duality, to a bijective map defined on the space of tempered distributions $\mathcal{S }'(\R^d)$. In case of a function of several variables, we place the symbol $\hat{\text{ }}$ over a given variable, in order to indicate that the Fourier transform was applied to it. 

We start by proving Theorem \ref{Reproducing_Formula_iff_1d_Isometry}. We show that the admissibility condition for the reproducing formula may be expressed in terms of  isometries from $L^2(\R)$ into $L^2\left(\R_+, \frac{ds}{s}\right)$.

\noindent
{\it Proof of Theorem  \ref{Reproducing_Formula_iff_1d_Isometry}.} Take tensor products $f,g \in L^2(\R^2)$, $f(x_1,x_2)=f_1(x_1)\,f_2(x_2)$,  $g(x_1,x_2)=g_1(x_1)\,g_2(x_2)$, with $f_1,f_2,g_1,g_2\in \mathcal{S}(\R)$. In the first step we express the inner products in terms of combined actions of the $L^1(\R)$ involution and dilation $h\mapsto s^{-1}\overline{h(-\cdot/s)}$, denoted as $\cdot^*_s$,
and the convolution on $\R$ represented as $*$  
\begin{align}
\int_{\R^2_+}&\langle f,\psi_{u,s} \rangle \langle \psi_{u,s},g \rangle \frac{du\,ds}{s^2}=\nonumber\\ 
&=\int_{\R^2_+} \int_\R f(\cdot,x_2)*\psi(\cdot^*_s,x_2)(u)\,dx_2
\int_\R \overline{g(\cdot,x_2)}*\overline{\psi(\cdot^*_s,x_2)}(u)\,dx_2 \,du\frac{ds}{s}.
\nonumber
\end{align}
 Representation of inner products as iterated integrals is justified by the fact, that for $u,s$ fixed,  functions $f(x_1,x_2)\,\psi_{u,s}(x_1,x_2)$, $g(x_1,x_2)\,\psi_{u,s}(x_1,x_2)$ are integrable with respect to $x_1,x_2$.
In the second step we apply Plancherel's formula with respect to $u$ and move the Fourier transform under the integral sign in order to get the following:
\begin{align}
\int_{\R^2_+}&\langle f,\psi_{u,s} \rangle \langle \psi_{u,s},g \rangle \frac{du\,ds}{s^2}=\nonumber\\ 
&=
\int_{\R^2_+} \int_\R f(\hat{\xi}, x_2)\overline{\psi(\widehat{s\xi}, x_2)}\,dx_2\,
\int_\R\overline{ g(\hat{\xi}, x_2)}\psi(\widehat{s\xi}, x_2)\,dx_2\, d\xi\frac{ds}{s}.
\nonumber
\end{align}
We use an approximation argument in order to justify the transition of the Fourier transform under the integral sign. We represent the square integrable kernel $\psi(\cdot^*_s,\cdot)$ defined on of $\R^2$, $s$ is fixed, as an infinite sum of orthogonal tensor products, we apply Plancherel's formula to finite sums, and then pass to $L^2(\R^2)$ norm limits in both expressions, the original one and the one obtained by an application of Plancherel's formula. In the third step we change the order of integration, we apply multiplicative invariance of the measure $\frac{ds}{s}$, we  denote by $\pm$ the sign of $\xi$,
and we get that 
\begin{align}
\int_{\R^2_+}&\langle f,\psi_{u,s} \rangle \langle \psi_{u,s},g \rangle \frac{du\,ds}{s^2}=\nonumber\\ 
&= \int_\R \hat{f_1}(\xi)\overline{\hat{g_1}(\xi)}
\int_0^\infty \int_\R f_2(x_2)\overline{\psi(\widehat{\pm s},x_2)}\,dx_2
\int_\R \overline{g_2(x_2)}\psi(\widehat{\pm s},x_2)\,dx_2 \frac{ds}{s}d\xi.
\label{reproduction_for_tensors}
\end{align}
Change of the order of integration is justified via polarization. We write the expressions under the integral signs as sums of the form we obtain for $f_1=g_1$, $f_2=g_2$. Non-negativity of the resulting terms allows us to apply Fubini's theorem.

If the system $\{\psi_{u,s}\}_{u\in \R, s>0}$ is reproducing, then, via formula 
(\ref{reproduction_for_tensors}), we conclude that the maps 
$f\mapsto \int_\R \overline{\psi\left(\widehat{\pm s},y\right)} f(y)\,dy$ restricted to $f\in \mathcal{S}(\R)$ preserve inner products. A standard density argument allows us to extend the statement to all $f\in L^2(\R)$. Conversely, if the maps $f\mapsto \int_\R \overline{\psi\left(\widehat{\pm s},y\right)} f(y)\,dy$ preserve inner products, then 
(\ref{reproduction_for_tensors}) allows us to conclude that for $f,g \in L^2(\R^2)$ being finite sums of tensor products of the form $f_1(x_1)\,f_2(x_2)$,  $g_1(x_1)\,g_2(x_2)$, with $f_1,f_2,g_1,g_2\in \mathcal{S}(\R)$ we have
$$
\int_{\R^2_+}\langle f,\psi_{u,s} \rangle \langle \psi_{u,s},g \rangle \frac{du\,ds}{s^2}=\langle f,g\rangle.
$$ 
Again, a standard density argument allows us to extend the equality to all $f,g\in L^2(\R^2)$.
\qed

The following representation of the inner product on $L^2(\R)$, valid for band limited functions, is the principal tool allowing us to make a transition from the continuous context of Theorem \ref{Reproducing_Formula_iff_1d_Isometry} to the discrete context of Theorem \ref{Reproducing_Discrete_System_iff_1d_Isometry}.

\begin{lem}\label{Inner_Product_for_BL} Let $f,g\in L^2(\R)$. Suppose that $\text{ess-supp}\,\hat{f}, \hat{g}\subset 
[-2^{-k-1}, 2^{-k-1}]$. Then
$$
\int_{-2^{-k-1}}^{2^{-k-1}}\hat{f}(\xi)\overline{\hat{g}(\xi)}\,
d\xi = 2^k\sum_{m\in \Z} f\left(2^km\right) \overline{ g\left(2^km\right)}.
$$
\end{lem}

\begin{proof}
Assume first that $k=0$. Let $\{e_n\}_{n\in \Z}$ be an orthonormal basis of the Hilbert space 
$$
\text{BL}=\{ f\in L^2(\R)\,|\,\text{ess-supp}\,\hat{f}\subset \left[-1/2, 1/2\right] \},$$ 
given by the formula $\hat{e}_n(\xi)=\chi_{\left[-1/2, 1/2\right]}(\xi)\,e^{-2\pi i n\xi}$.
Observe that for $f\in \text{BL}$ $\langle f,e_n\rangle=f(n)$. Indeed, the Fourier inversion formula holds pointwise for functions in $\mathcal{S}(\R)\cap \text{BL}$. Moreover, convergence in $\text{BL}$ implies pointwise convergence, therefore a standard density arguments shows that the identity holds for all functions in $\text{BL}$. An applications of Plancherel's formula gives
$$
\sum_{n\in \Z}f(n)\overline{g(n)}=\sum_{n\in \Z}\langle f,e_n \rangle
\overline{\langle g,e_n \rangle}=\langle f,g \rangle = \langle \hat{f},\hat{g} \rangle = \int_{-1/2}^{1/2}\hat{f}(\xi)\overline{\hat{g}(\xi)}d\xi.
$$
A change of variables allows us to derive the formula for any $k\in \Z$. Indeed
\begin{align}
\int_{-2^{-k-1}}^{2^{-k-1}}&\hat{f}(\xi)\overline{\hat{g}(\xi)}\,d\xi =
\int_{-1/2}^{1/2}\hat{f}\left(\frac{\xi}{2^k}\right)\overline{\hat{g}
\left(\frac{\xi}{2^k}\right)}\,\frac{d\xi}{2^k}=\nonumber \\
&=2^{-k}\int_{-1/2}^{1/2}\left( f_{2^{-k}}\right)^{\wedge}(\xi)
\overline{\left( g_{2^{-k}}\right)^{\wedge}(\xi)}\, d\xi=
2^{-k}\sum_{n\in \Z} 2^kf(2^kn)2^k\overline{g(2^kn)},
\nonumber
\end{align}
where the subscript applied to a function  denotes the $L^1(\R)$ dilation, i.e.
$h_s(x)=\frac{1}{s}h(\frac{x}{s})$.
\end{proof}

We are ready to prove  Theorem \ref{Reproducing_Discrete_System_iff_1d_Isometry}.
The method of the proof of Theorem \ref{Reproducing_Formula_iff_1d_Isometry} combined with the formula of Lemma \ref{Inner_Product_for_BL} allow us to formulate necessary and sufficient conditions for the discretized system to be a Parseval frame. Again the admissibility condition is expressed in terms of isometries, now from $L^2(\R)$ into $l^2(\Z)$.

\noindent
{\it Proof of Theorem  \ref{Reproducing_Discrete_System_iff_1d_Isometry}.}
Take tensor products $f,g \in L^2(\R^2)$, $f(x_1,x_2)=f_1(x_1)\,f_2(x_2)$,  $g(x_1,x_2)=g_1(x_1)\,g_2(x_2)$, with $f_1,f_2,g_1,g_2\in \mathcal{S}(\R)$. In the first step we express inner products $\langle f,\psi_{k,m} \rangle, \langle g, \psi_{k,m} \rangle $ as iterated integrals 
\begin{align}
\sum_{k,m}&\langle f,\psi_{k,m} \rangle \langle \psi_{k,m},g \rangle =\nonumber\\ 
&=\sum_{k,m} 2^k \int_\R f(\cdot,x_2)*\psi(\cdot^*_{2^k},x_2)(2^km)\,dx_2
\int_\R \overline{g(\cdot,x_2)}*\overline{\psi(\cdot^*_{2^k},x_2)}(2^km)\,dx_2,
\nonumber
\end{align}
where $\cdot^*_{2^k}$ is the combined action of the $L^1(\R)$ involution and dilation $h\mapsto 2^{-k/2}\overline{h(-\cdot/2^k)}$, and $*$ is the convolution on $\R$.
The transition to iterated integrals is justified by the integrability of $f(x_1,x_2)\,\psi_{k,m}(x_1,x_2)$, 
$g(x_1,x_2)\,\psi_{k,m}(x_1,x_2)$ with respect to $x_1,x_2$, parameters $k,m$ are fixed.
In the second step we apply Lemma \ref{Inner_Product_for_BL}, we obtain that
\begin{align}
\sum_{k,m}&\langle f,\psi_{k,m} \rangle \langle \psi_{k,m},g \rangle =\nonumber\\ 
&=\sum_k \int_\R  \int_\R f(\hat{\xi}, x_2)\overline{\psi(\widehat{2^k\xi}, x_2)}\,dx_2\,
\int_\R\overline{ g(\hat{\xi}, x_2)}\psi(\widehat{2^k\xi}, x_2)\,dx_2\, d\xi.
\nonumber
\end{align}
 The usage of Lemma \ref{Inner_Product_for_BL} is justified by the fact that for almost every $x_2\in \R$ we have $\text{ess-supp}\,\psi(\cdot^*_{2^k},x_2)\subset \left[-2^{-k-1}, 2^{-k-1}\right]$. We represent the square integrable kernel $\psi(\cdot^*_{2^k},\cdot)$, defined on $\R^2$, $k$ is fixed, as an infinite sum of orthogonal tensor products of functions, band limited, with respect to the first coordinate, and square integrable, with respect to the second coordinate. Then, we apply Lemma \ref{Inner_Product_for_BL} to finite sums, and next we pass to norm limits in both expressions, the original one, and the one obtained by an application of  Lemma \ref{Inner_Product_for_BL}. In the next step we change  the order of summation and integration and we get 
\begin{align}
\sum_{k,m}&\langle f,\psi_{k,m} \rangle \langle \psi_{k,m},g \rangle =\nonumber\\ 
&=\int_\R \hat{f_1}(\xi)\overline{\hat{g_1}(\xi)}
\sum_k \int_\R f_2(x_2)\overline{\psi(\widehat{2^k \xi},x_2)}\,dx_2
\int_\R \overline{g_2(x_2)}\psi(\widehat{2^k \xi},x_2)\,dx_2 \,d\xi.
\label{discrete_reproduction_for_tensors}
\end{align}
An application of polarization formula produces non-negative terms, therefore Fubini's theorem applies.

If the system $\{\psi_{k,m}\}_{k,m \in \Z}$ is reproducing, then, via formula 
(\ref{discrete_reproduction_for_tensors}), we have  
$$\int_\R|\hat{f}_1(\xi)|^2\sum_k \left|\int_\R f_2(x_2) 
\overline{\psi(\widehat{2^k \xi},x_2)}\,dx_2\,\right|^2d\xi
=||f_1||^2||f_2||^2,
$$
for all $f_1,f_2\in \mathcal{S}(\R)$,
and therefore for every $f\in \mathcal{S}(\R)$
\begin{equation}
\sum_k \left|\int_\R 
\overline{\psi(\widehat{2^k \xi},y)} f(y)\,dy\,\right|^2 = ||f||^2
\label{discrete_a_e_norm_condition}
\end{equation}
for almost every $\xi \in \R$. A standard density argument, making use of the convergence in the mixed norm space $L^\infty(l^2)$, allows us to conclude that for every 
$f\in L^2(\R)$
(\ref{discrete_a_e_norm_condition})
holds for almost every $\xi \in \R$. The fact that 
for every pair $f,g\in L^2(\R)$ the equality (\ref{discrete_isometry})
holds for almost every $\xi\in \R$ follows by polarization. 
Conversely, if for every pair $f,g\in L^2(\R)$ the equality (\ref{discrete_isometry})
holds for almost every $\xi\in \R$, then
(\ref{discrete_reproduction_for_tensors}) allows us to conclude that for $f,g \in L^2(\R^2)$ being finite sums of tensor products of the form $f_1(x_1)\,f_2(x_2)$,  $g_1(x_1)\,g_2(x_2)$, with $f_1,f_2,g_1,g_2\in \mathcal{S}(\R)$ we have
$$
\sum_{k,m}\langle f,\psi_{k,m} \rangle \langle \psi_{k,m},g \rangle =\langle f,g\rangle.
$$
Again, a standard density argument allows us to extend the equality to all $f,g\in L^2(\R^2)$.
\qed

The following lemma summarizes the basic properties of the generating functions $\psi^D$. 
\begin{lem}\label{Basic_psi_^D_properties} 
Let $\psi^D$ be the generating function defined in
(\ref{2d_generating_function}). Then
\begin{align}
\text{(i)}\, &\text{the sum } (\ref{2d_generating_function}) \text{ representing } \psi^D(\hat{s},y)\, \text{ consists of a single term }
f_{D(k,l)}(s)e_{k,l}(y), 
\nonumber \\
&\text{for } s\in [-1/2,0)\cup (0,1/2] \text{, with the unique }k,l 
\text{ satisfying } 
\nonumber \\
&s\in (2^{-D(k,l)-1}, 2^{-D(k,l)}], \text{ and it contains no non-zero terms for }
\nonumber \\
&s\notin [-1/2,0)\cup (0,1/2],
\nonumber \\
\text{(ii)}\, &\text{ess-supp}\,\psi^D(\hat{\cdot},y)\subset [-1/2,1/2] \text{ for every }y\in \R ,
\nonumber \\
\text{(iii)}\, &\int_{\R^2}\left| \psi^D(\hat{s},y) \right|^2dy\, ds =1,
\nonumber \\
\text{(iv)}\,&S_N^D(s,y)=\sum_{|k|\le N,|l|\le N}f_{D(k,l)}(s)e_{k,l}(y)
\text{ converges to }\psi^D(\hat{s},y) \text{ in } L^2(\R^2) \text{ as }
N \rightarrow \infty.
\nonumber
\end{align}
\end{lem} 

\begin{proof} Part (i) follows directly from definitions. 
Part (ii) is justified by the fact that $\text{ess-supp}\,\psi^D(\hat{\cdot},y) \subset  
\text{ess-supp}\,\sum_{k,l}f_{D(k,l)} = [-1/2,1/2]$.
We observe that the building blocks 
$\{f_{D(k,l)}e_{k,l}(y)\}_{k,l\in Z}$
of formula (\ref{2d_generating_function})
have pairwise disjoint supports. We conclude that
\begin{equation}
\left|\psi^D(\hat{s},y)\right|=\sum_{k,l\in \Z}f_{D(k,l)}(s)\left| e_{k,l}(y)\right|
=\sum_{k,l\in \Z}\chi_{(2^{-D(k,l)-1}, 2^{-D(k,l)}]}(|s|)
\chi_{(k,k+1]}(y).
\label{psi_D_abs_sum}
\end{equation} 
The operation of shifting pairs of rectangles
$$
\chi_{(2^{-D(k,l)-1}, 2^{-D(k,l)}]}(|s|) \chi_{(k,k+1]}(y)
$$
of formula (\ref{psi_D_abs_sum}) to position $k=0$, i.e. to 
$$
\chi_{(2^{-D(k,l)-1}, 2^{-D(k,l)}]}(|s|) \chi_{(0,1]}(y)
$$
is measure preserving, therefore out of formula (\ref{psi_D_abs_sum})
we infer that
$$
\int_{\R^2}\left| \psi^D(\hat{s},y) \right|^2dy\, ds 
= \int_{-1/2}^{1/2}\int_0^1 dy\, ds = 1.
$$
This proves part (iii). Part (iv) follows via the technique applied to part (iii).
We observe that as $N$ gets large enough, only the dyadic intervals representing 
$f_{D(k,l)}$, which are close to $0$, show up in the expression controlling 
the $L^2(\R^2)$ norm of $\psi^D(\hat{\cdot},\cdot)-S_N^D$. 
\end{proof}

With the background results presented so far, we are ready now to draw conclusions for the generating functions $\psi^D$, i.e. to prove Corollaries \ref{2d_Reproducing_System}, \ref{2d_Orthonormal_Basis}.

\noindent
{\it Proof of Corollary \ref{2d_Reproducing_System}.}
We have checked in Lemma \ref{Basic_psi_^D_properties} that 
$\psi^D\in L^2(\R^2)$. We know via Theorem
\ref{Reproducing_Formula_iff_1d_Isometry} that it is enough to show that the
$\pm$ maps
$$
f\mapsto \int_\R \overline{\psi^D\left(\widehat{\pm s},y\right)} f(y)\,dy,
$$
from $L^2(\R)$ into $L^2(\R_+,\frac{ds}{s})$, preserve inner products.
We have verified in Lemma \ref{Basic_psi_^D_properties} that for both $\pm s$, where $s>0$ is fixed, at most one of the terms of 
$$
\sum_{k,l}f_{D(k,l)}(\pm s)e_{k,l}(y)
$$  
is non-zero. Therefore, for every $s>0$, we may switch over summation and integration and we obtain
\begin{equation}
\int_\R \overline{\psi^D\left(\widehat{\pm s},y\right)} f(y)\,dy=
\sum_{k,l}f_{D(k,l)}(\pm s)\langle f,e_{k,l}\rangle
\label{k_l_expansion_for_f}
\end{equation}
and
\begin{equation}
\int_\R \overline{\psi^D\left(\widehat{\pm s},y\right)} g(y)\,dy=
\sum_{k,l}f_{D(k,l)}(\pm s)\langle g,e_{k,l}\rangle.
\label{k_l_expansion_for_g}
\end{equation}
The systems $\{d_f^{-1}f_m(\pm s)\}_{m\ge 1}$ are orthonormal in 
$L^2(\R_+,\frac{ds}{s})$. Formulas (\ref{k_l_expansion_for_f}),
(\ref{k_l_expansion_for_g}) allow us to conclude that
$$
\left < \int_\R \overline{d_f^{-1}\psi^D\left(\widehat{\pm s},y\right)} f(y)\,dy,
\int_\R \overline{d_f^{-1}\psi^D\left(\widehat{\pm s},y\right)} g(y)\,dy
\right > _{L^2(\R_+,\frac{ds}{s})}=
$$
$$
=\sum_{k,l}\langle f,e_{k,l}\rangle \overline{\langle g,e_{k,l}\rangle}
=\langle f,g \rangle,
$$
and this finishes the proof.
\qed

\noindent
{\it Proof of Corollary \ref{2d_Orthonormal_Basis}.}
We have verified in Lemma \ref{Basic_psi_^D_properties} that 
$\psi^D\in L^2(\R^2)$ and that for every $y \in \R$ $\text{ess-supp}\, \psi^D(\hat{\cdot},y)\subset [-1/2,1/2]$. We know via Theorem \ref{Reproducing_Discrete_System_iff_1d_Isometry},
that in order to show that the system 
$\left\{\psi_{k,m}^D\right\}_{k,m\in \Z}$
is reproducing, it is enough to show that for every pair $f,g\in L^2(\R)$ the equality
$$
\langle f,g \rangle = \sum_k \int_\R 
\overline{\psi^D\left(\widehat{2^k \xi},y\right)} f(y)\,dy
\int_\R \psi^D\left(\widehat{2^k \xi},y\right) \overline{g(y)}\,dy
$$
holds for almost every $\xi\in \R$. 
Support properties of functions $f_{D(r,s)}$, $f_{D(r',s')}$ imply that for every $\xi \ne 0$
\begin{align}
\sum_k &\int_\R 
\overline{\psi^D\left(\widehat{2^k \xi},y\right)} f(y)\,dy
\int_\R \psi^D\left(\widehat{2^k \xi},y\right) \overline{g(y)}\,dy
\nonumber \\
&=\sum_k\sum_{\substack{r,s \\ r',s'}}f_{D(r,s)}(2^k\xi)f_{D(r',s')}(2^k\xi)
\langle f,e_{r,s}\rangle \overline{\langle g,e_{r,s}\rangle}
\nonumber \\
&= \sum_{r,s}\langle f,e_{r,s}\rangle \overline{\langle g,e_{r,s}\rangle}
=\langle f, g \rangle.
\nonumber
\end{align}
Changes of the order of summation with respect to $r,s$ and $r',s'$, with the  integration with respect to $y$, 
are justified by Lemma \ref{Basic_psi_^D_properties}, i.e. by the fact that the the sums with respect   
to $r,s$ and $r',s'$ consist of at most one term. We have proved completeness of the system $\left\{\psi_{k,m}^D\right\}_{k,m\in \Z}$.

Now we prove the fact that the functions of the system $\left\{\psi_{k,m}^D\right\}_{k,m\in \Z}$ are pairwise orthogonal. 
The first step is based on an application of Plancherel's formula with respect to $x_1$ 
\begin{align}
\langle \psi^D_{k,m}&, \psi^D_{k',m'}\rangle = \int_{\R^2}
2^{-k/2}\psi^D \left( \frac{x_1-2^km}{2^k},x_2\right)
2^{-k'/2}\overline{\psi^D \left( \frac{x_1-2^{k'}m'}{2^{k'}},x_2\right)}\,dx_1\,dx_2 =
\nonumber \\
&=\int_{\R^2}
2^{k/2}e^{-2\pi i 2^km\xi}\psi^D \left( \widehat{2^k\xi},x_2\right)
2^{-k'/2}e^{2\pi i 2^{k'}m'\xi}
\overline{\psi^D \left( \widehat{2^{k'}\xi},x_2\right)}\,d\xi\,dx_2
\nonumber
\end{align}
The transition to iterated integrals is justified by the fact that the function under the integral sign is integrable.
In the second step we change the order of integration and we apply the definition of $\psi^D$
\begin{align}
&\int_{\R^2}
2^{k/2}e^{-2\pi i 2^km\xi}\psi^D \left( \widehat{2^k\xi},x_2\right)
2^{-k'/2}e^{2\pi i 2^{k'}m'\xi}
\overline{\psi^D \left( \widehat{2^{k'}\xi},x_2\right)}\,dx_2\,d\xi
\nonumber \\
&=\sum_{\substack{r,s \\ r',s'}} 
\int_{\R^2}2^{k/2}e^{-2\pi i 2^km\xi}f_{D(r,s)}(2^k\xi)e_{r,s}(x_2)
2^{-k'/2}e^{2\pi i 2^{k'}m'\xi}f_{D(r',s')}(2^{k'}\xi)\overline{e_{r',s'}(x_2)}
\,dx_2\,d\xi
\nonumber
\end{align}
The change of order of integration is justified by the integrability of the function under the integral sign. 
Moving out the summations with respect to $r,s$ and $r', s'$ is justified by the $L^2(\R^2)$ norm convergence of the expansion of (\ref{2d_generating_function}) stated in Lemma \ref{Basic_psi_^D_properties}. In the third step
we apply the orthogonality of the system $\{e_{r,s}\}_{r,s\in \Z}$
and the fact that the sets $\{\xi \in \R\,|\,|\xi|\in (2^{-D(r,s)-1-k},2^{-D(r,s)-k}]\}$, $k\in \Z$,
are pairwise disjoint
\begin{align}
&\sum_{r,s}\int_{\R}
2^{k/2}e^{-2\pi i 2^km\xi}f_{D(r,s)}(2^k\xi)
2^{-k'/2}e^{2\pi i 2^{k'}m'\xi}f_{D(r,s)}(2^{k'}\xi)\,d\xi
\nonumber \\
&=\delta_{k,k'}2^k\int_{-2^{-k-1}}^{2^{-k-1}}
e^{2\pi i 2^k (m'-m)\xi}\,d\xi = \delta_{k,k'}\,\delta_{m,m'}.
\nonumber
\end{align}
We have verified both completeness and orthogonality of the system $\left\{\psi_{k,m}^D\right\}_{k,m\in \Z}$, therefore the fact that  it constitutes an orthonormal basis follows.
\qed

{\bf Acknowledgements} 

The authors would like to thank CIRM (Luminy, Marseille, France), Jean-Morlet Semester, Hans Feichtinger (Chair) and Bruno Torr\'esani  (Local project leader), where the work on this paper began, and the University of P\'ecs (P\'ecs, Hungary), where it was continued. The authors are also thankful for instructive conversations with Przemyslaw  Wojtaszczyk.

The present scientific contribution is dedicated to the $650$th anniversary of the foundation of the University
of P\'ecs, Hungary.

\end{document}